\documentclass[12pt,reqno]{amsart}
\numberwithin{equation}{section}
\usepackage[T1]{fontenc}
\usepackage{lmodern}
\usepackage{amsmath,amssymb,mathtools,mathrsfs}
\usepackage[margin=1.08in]{geometry}
\usepackage{microtype}
\usepackage{enumitem}
\usepackage{needspace}
\usepackage{float}
\usepackage{tikz-cd}
\usetikzlibrary{arrows}
\usepackage{xcolor}
\usepackage[colorlinks=true,linkcolor=blue!55!black,citecolor=blue!55!black,urlcolor=blue!55!black]{hyperref}
\hypersetup{pdftitle={Spectra and Support Theory for Wide Subcategories in n-Abelian Categories}}

\newtheorem{theorem}{Theorem}[section]
\newtheorem{proposition}[theorem]{Proposition}
\newtheorem{lemma}[theorem]{Lemma}
\newtheorem{corollary}[theorem]{Corollary}
\theoremstyle{definition}
\newtheorem{definition}[theorem]{Definition}

\theoremstyle{remark}
\newtheorem{remark}[theorem]{Remark}

\newcommand{\Hom}{\operatorname{Hom}}
\newcommand{\End}{\operatorname{End}}
\newcommand{\Ext}{\operatorname{Ext}}
\newcommand{\add}{\operatorname{add}}
\newcommand{\modcat}{\operatorname{mod}}
\newcommand{\Ob}{\operatorname{obj}}
\newcommand{\Spec}{\operatorname{Spec}}
\newcommand{\Supp}{\operatorname{Supp}}
\newcommand{\Wide}{\mathbf{Wide}}
\newcommand{\Rad}{\mathbf{Rad}}
\newcommand{\Param}{\mathbf{Param}}
\newcommand{\Prime}{\mathbf{Prime}}
\newcommand{\Spcl}{\mathbf{Spcl}}
\newcommand{\Ccl}{\mathscr{C}_{cl}}
\newcommand{\K}{\mathbf{k}}
\newcommand{\M}{\mathcal M}
\newcommand{\W}{\mathcal W}
\newcommand{\X}{\mathcal X}
\newcommand{\Pp}{\mathcal P}
\newcommand{\Qq}{\mathcal Q}
\DeclareMathOperator{\rad}{rad}
\DeclareMathOperator{\wide}{wide}

\title[Spectra and Support Theory for Wide Subcategories in \(n\)-Abelian Categories]{Spectra and Support Theory for Wide Subcategories in \(n\)-Abelian Categories}

\author{Zheng Xin}
\address{School of Mathematics and Statistics, Shaanxi Normal University, Xi'an 710062, China}
\email{xinzheng1314@yeah.net}

\date{}

\begin{document}
	
	\begin{abstract}
		We develop a support theory for wide subcategories of essentially small
		Krull--Schmidt \(n\)-abelian categories. We construct spectra associated
		with families of wide subcategories and establish the corresponding
		support and radical theories. In particular, we characterize prime wide
		subcategories as the completely meet-irreducible elements of the lattice
		of wide subcategories and prove that every wide subcategory is recovered
		from the prime wide subcategories containing it. We further show that any
		classifying support datum is recovered, together with its object supports,
		from the canonical prime spectrum. Finally, we provide a \(2\)-abelian
		example admitting no classifying support datum.
	\end{abstract}
	
	\maketitle
	\setcounter{tocdepth}{1}
	\tableofcontents
	
	\section{Introduction}
	
	Wide subcategories of abelian categories are full subcategories closed under
	kernels, cokernels, and extensions. Their classification connects homological
	closure conditions with the geometry of supports and the order structure of
	subcategory lattices. Support-theoretic classifications for module categories
	were obtained by Hovey for suitable commutative rings \cite{H01} and by
	Takahashi for commutative noetherian rings \cite[Theorem~A]{T08}.
	In representation theory, wide subcategories are also related to noncrossing
	partitions, torsion classes, and collections of bricks
	\cite{IT09,MS17,AP22,E21}.
	
	Balmer's tensor-ideal spectrum \cite{B05} and Matsui's spectra of prime
	thick subcategories \cite{M21} provide topological models for subcategory
	classification. Related constructions occur in extriangulated and tensor
	$n$-angulated categories \cite{TZH26,T26}. Building on the spectrum and
	support theory for wide subcategories of abelian categories developed
	in \cite{TWZ26}, we extend this framework to
	essentially small Krull--Schmidt $n$-abelian categories.
	
	The higher setting originates in Iyama's higher-dimensional
	Auslander--Reiten theory \cite{I07}, Jasso's $n$-abelian categories
	\cite{J16}, and the study of their wide subcategories by Herschend,
	J{\o}rgensen, and Vaso \cite{HJV20}. For $n\geq2$, $n$-kernels,
	$n$-cokernels, and representatives of $n$-extension classes need not
	be unique up to isomorphism, so the existence of suitable representatives
	in individual wide subcategories does not by itself establish closure
	under intersections. The minimal-representative criterion of
	\cite[Corollary~2.6]{HJ21} resolves this difficulty. 
	
	Let $\M$ be an essentially small Krull--Schmidt $n$-abelian category,
	and let $\Wide_n(\M)$ denote its poset of wide subcategories
	(Definition~\ref{def:wide}), ordered by inclusion.
	A proper wide subcategory is \emph{prime} if the wide subcategories
	strictly containing it have a least element
	\cite[Remark~2.9]{M23}.
	Write $\Prime_n(\M)$ for the set of primes and
	$\Spec_{\Prime_n(\M)}(\M)$ for this set with the topology
	whose closed sets are arbitrary intersections of the supports
	\[
	\Supp(M)=\{\Pp\in\Prime_n(\M)\mid M\notin\Pp\}.
	\]
	
	For a topological space $X$ and an assignment $\sigma$, write
	\[
	\sigma(\X)=\bigcup_{M\in\X}\sigma(M),
	\qquad
	\sigma^{-1}(V)=\{M\in\M\mid\sigma(M)\subseteq V\}.
	\]
	An $n$-support datum $(X,\sigma)$ is \emph{classifying} if
	$X$ is noetherian and sober and these maps give mutually inverse
	order isomorphisms between $\Wide_n(\M)$ and the
	specialization-closed subsets of $X$.
	
	\begin{theorem}\label{thm:intro-main}
	Let $\M$ be an essentially small Krull--Schmidt $n$-abelian category.
	\begin{enumerate}[label=\textup{(\arabic*)},leftmargin=*]
		\item Let $X$ be a topological space and let $\sigma$ be an
		isomorphism-invariant assignment of closed subsets of $X$
		to objects of $\M$. Then $(X,\sigma)$ is an $n$-support datum
		if and only if $\sigma^{-1}(\omega)$ is wide for every subset
		$\omega\subseteq X$.
		
		\item Every wide subcategory of $\M$ is the intersection of
		the prime wide subcategories containing it, with the empty
		intersection understood as $\M$.
		
		\item For every classifying $n$-support datum $(X,\sigma)$
		on $\M$, the map
		\[
		\phi\colon X\longrightarrow\Spec_{\Prime_n(\M)}(\M),
		\qquad
		\phi(x)=\{M\in\M\mid x\notin\sigma(M)\}
		\]
		is a homeomorphism satisfying
		$\phi^{-1}(\Supp(M))=\sigma(M)$ for every $M\in\M$.
	\end{enumerate}
	\end{theorem}
	For $n=1$, our constructions recover the abelian framework
	of \cite{TWZ26}, with the least-element convention for primes;
	the Krull--Schmidt hypothesis is unnecessary in this case.
	For $n=2$, we revisit the cluster-tilting subcategory over
	$\K(1\to2\to3)/\langle\beta\alpha\rangle$
	from \cite[Section~7]{HJV20} and \cite[Section~7.1]{F26},
	and compute its prime spectrum and supports.
	Its eight wide subcategories have distinct supports, whereas
	its prime spectrum has twenty-five specialization-closed subsets.
	By Theorem~\ref{thm:intro-main}\textup{(3)}, no classifying
	support datum exists.

	\noindent\textbf{Conventions and notation.}
	Throughout, $n\geq1$ and $\K$ is a field.
	All subcategories are full and closed under isomorphisms.
	We write $\add\mathcal S$ for the closure of $\mathcal S$
	under finite direct sums, direct summands, and isomorphisms.
	The specialization order is
	$x\leq_{\mathrm{sp}}y$ if $x\in\overline{\{y\}}$.
	
	\section*{Acknowledgments}
	I am grateful to my supervisor, Prof. Wen Chang,
	for his academic guidance and support.

	\section{Wide subcategories of \texorpdfstring{$n$}{n}-abelian categories}\label{sec:prelim}
	
	We recall higher kernels, cokernels, and extensions, and then use minimal
	representatives to prove that intersections of wide subcategories are wide.
	The definitions and the minimal-representative criterion are taken from
	\cite{J16,HJV20,HJ21}.
	
	Let $\mathcal C$ be an additive category and let $f:X^0\to X^1$ be a morphism. An \emph{$n$-cokernel} of $f$ is a sequence
	\[
	X^1\xrightarrow{d^1}X^2\longrightarrow\cdots\longrightarrow X^{n}\xrightarrow{d^n}X^{n+1}
	\]
	such that, for every $Y\in\mathcal C$, the induced sequence
	\[
	0\to\mathcal C(X^{n+1},Y)\to\mathcal C(X^n,Y)\to\cdots\to
	\mathcal C(X^1,Y)\to\mathcal C(X^0,Y)
	\]
	is exact. The notion of an \emph{$n$-kernel} is defined dually.
	
	A complex
	$X^0\longrightarrow X^1\longrightarrow\cdots\longrightarrow X^n\longrightarrow X^{n+1}$
	is \emph{$n$-exact} if its first $n$ morphisms form an $n$-kernel of the last morphism and its last $n$ morphisms form an $n$-cokernel of the first morphism. Such a complex is also called an \emph{$n$-extension} of $X^{n+1}$ by $X^0$.
	
	An \emph{$n$-abelian category} is an idempotent-complete additive
	category in which every morphism has an $n$-kernel and an
	$n$-cokernel, subject to the following two axioms: a monomorphism
	together with any of its $n$-cokernels is $n$-exact, and an
	epimorphism together with any of its $n$-kernels is $n$-exact, see \cite[Definition~3.1]{J16}. For $n=1$, this is the usual
	notion of an abelian category.
	
	Two $n$-extensions with the same endpoints are \emph{equivalent}
	if there is a chain map between them that is the identity at both
	endpoints. This is an equivalence relation in an $n$-abelian
	category, see \cite[Proposition~4.10]{J16} and
	\cite[Definition~2.2]{HJ21}.
	
	\begin{definition}\cite[Definition~2.3]{HJ21}\label{def:wide}
		Let $\M$ be an $n$-abelian category. A full additive subcategory
		$\W\subseteq\M$ is \emph{wide} if it is closed under finite direct sums
		and direct summands and satisfies the following conditions.
		\begin{enumerate}[label=\textup{(W\arabic*)},leftmargin=*]
			\item Every morphism in $\W$ admits an $n$-kernel and an $n$-cokernel
			in $\M$ all of whose terms belong to $\W$. \label{W1}
			\item Every $n$-extension in $\M$ with end terms in $\W$ is equivalent
			to an $n$-extension whose middle terms all belong to $\W$. \label{W2}
		\end{enumerate}
		We write $\Wide_n(\M)$ for the poset of wide subcategories ordered by inclusion.
	\end{definition}
	
	A category is \emph{Krull--Schmidt} if every object is a finite
	direct sum of objects with local endomorphism rings. Assume for the
	rest of this section that $\M$ is Krull--Schmidt. Its categorical
	Jacobson radical is
	\[
	\rad_{\M}(U,V)=
	\{a:U\to V\mid 1_U-ba\text{ is invertible for every }b:V\to U\}.
	\]
	This ideal of morphisms is distinct from the radical of a wide
	subcategory introduced in Section~\ref{sec:radical}.
	
	\begin{definition}\cite[Definition~2.5]{HJ21}\label{def:minimal}
		Let $f:M\to N$ be a morphism in $\M$.
		\begin{enumerate}[label=\textup{(\arabic*)},leftmargin=*]
			\item An $n$-kernel $K_n\to\cdots\to K_1\to M$ of $f$ is
			\emph{minimal} if the maps $K_i\to K_{i-1}$ belong to
			$\rad_{\M}$ for $2\leq i\leq n$.
			\item An $n$-cokernel $N\to C_1\to\cdots\to C_n$ of $f$ is
			\emph{minimal} if the maps $C_i\to C_{i+1}$ belong to
			$\rad_{\M}$ for $1\leq i<n$.
			\item An $n$-extension $X\to E_n\to\cdots\to E_1\to Y$ is
			\emph{minimal} if its maps $E_i\to E_{i-1}$ belong to
			$\rad_{\M}$ for $2\leq i\leq n$.
		\end{enumerate}
	\end{definition}
	
	The maps incident with the fixed endpoints are not required to be
	radical. For $n=1$, all three radical conditions are empty.
	
	\begin{lemma}\cite[Proposition~2.4 and Definition~2.5]{HJ21}\label{lem:additive-minimal-representatives}
		In a Krull--Schmidt \(n\)-abelian category, the following properties hold:
		\begin{enumerate}
			\item Every morphism has a minimal \(n\)-kernel and a minimal
			\(n\)-cokernel.
			
			\item Every \(n\)-extension class has a minimal representative.
			
			\item Each minimal representative is a direct summand, as a complex,
			of every corresponding representative and is unique up to isomorphism.
			Moreover, the comparison morphisms and isomorphisms may be chosen to
			be the identity on the prescribed endpoints.
		\end{enumerate}
	\end{lemma}
	
	The proof of \cite[Proposition~2.4]{HJ21} does not require
	$\K$-linearity or Hom-finiteness, so it applies in the present generality.
	For $n=1$, the fixed-endpoint assertions follow from uniqueness of
	kernels and cokernels and the short five lemma. For $n\geq2$,
	the comparison maps in that proof fix the prescribed endpoints. If $a$ and $b$
	are the comparison maps in opposite directions, their composite $ba$
	on a minimal representative is a degreewise automorphism. Replacing
	$b$ by $(ba)^{-1}b$ gives a retraction still fixing the endpoints.
	When both representatives are minimal, the same argument applies to
	$ab$, and $a$ is an isomorphism fixing the endpoints.
	
	The direct-summand property in Lemma~\ref{lem:additive-minimal-representatives}
	gives the following characterization of wide subcategories.
	
	\begin{lemma}\cite[Corollary~2.6]{HJ21}\label{lem:minimal-characterization}
		Let $\M$ be a Krull--Schmidt $n$-abelian category and let $\W\subseteq\M$ be a subcategory closed under finite direct sums and direct summands.  Then $\W$ is wide if and only if:
		\begin{enumerate}[label=\textup{(\arabic*)},leftmargin=*]
			\item for every morphism $f:M\to N$ in $\W$, all terms of a (hence any) minimal $n$-kernel and minimal $n$-cokernel of $f$ belong to $\W$;
			\item for every $X,Y\in\W$, all middle terms of every minimal $n$-extension of $Y$ by $X$ belong to $\W$.
		\end{enumerate}
	\end{lemma}
	
	Lemma~\ref{lem:minimal-characterization} replaces the existence conditions in the definition of a wide
	subcategory by closure under minimal representatives. This permits us to take arbitrary intersections.
	
	\begin{proposition}\label{prop:intersection}
		Let $\M$ be a Krull--Schmidt $n$-abelian category and let
		$\{\W_\lambda\}_{\lambda\in\Lambda}$ be a family of wide subcategories.
		Then $\bigcap_{\lambda\in\Lambda}\W_\lambda$ is wide. The empty intersection is understood to be $\M$.
	\end{proposition}
	
	\begin{proof}
		Set $\W=\bigcap_{\lambda\in\Lambda}\W_\lambda$. Then $\W$ is full and
		closed under finite direct sums and direct summands. Let $f:M\to N$ be
		a morphism in $\W$, and choose a minimal $n$-kernel and a minimal $n$-cokernel of
		$f$ in $\M$. Since $M,N\in\W_\lambda$ for all $\lambda\in\Lambda$,
		Lemma~\ref{lem:minimal-characterization} shows that every term of these
		minimal complexes lies in each $\W_\lambda$, and hence in $\W$.
		
		Let $X,Y\in\W$, and let an arbitrary minimal $n$-extension of $Y$ by $X$ be given.
		Applying Lemma~\ref{lem:minimal-characterization} once again, we see
		that all middle terms of this extension lie in each $\W_\lambda$, and
		therefore in $\W$. Hence $\W$ satisfies the criterion of
		Lemma~\ref{lem:minimal-characterization}, and so $\W$ is wide.
		
		If $\Lambda=\varnothing$, then $\W=\M$, which is wide.
	\end{proof}
	\begin{remark}
		Intersection closure is not immediate from the definition when \(n\ge2\),
		since the higher kernels, cokernels, and extension representatives need not
		be unique up to isomorphism. Minimal representatives provide a common
		representative in every member of the family and hence establish this closure.
	\end{remark}
	\begin{corollary}\label{cor:complete-lattice}
		Let $\M$ be an essentially small Krull--Schmidt $n$-abelian category.
		Then $\Wide_n(\M)$ is a complete lattice. For every family
		$\{\W_\lambda\}$ of wide subcategories,
		\[
		\bigwedge_\lambda\W_\lambda
		=
		\bigcap_\lambda\W_\lambda,
		\qquad
		\bigvee_\lambda\W_\lambda
		=
		\wide_{\M}\!\left(\bigcup_\lambda\W_\lambda\right),
		\]
		where $\wide_{\M}(\mathcal S)$ denotes the least wide
		subcategory of $\M$ containing $\mathcal S$.
	\end{corollary}
	
	\begin{proof}
		Since $\M$ is essentially small, $\Wide_n(\M)$ is a set.
		Proposition~\ref{prop:intersection} shows that arbitrary meets exist.
		For a class $\mathcal S$ of objects of $\M$, the collection of wide
		subcategories containing $\mathcal S$ is nonempty, as it contains
		$\M$. By Proposition~\ref{prop:intersection}, their intersection is
		wide and is the least wide subcategory containing $\mathcal S$.
		Taking $\mathcal S=\bigcup_\lambda\W_\lambda$ gives the formula for
		joins.
	\end{proof}
	
	We will also need closure under directed unions to separate a wide
	subcategory from an object outside it. Recall that a family
	\(\{\W_i\}_{i\in I}\) is directed under inclusion if for any
	\(i,j\in I\) there exists \(k\in I\) such that
	$\W_i,\W_j\subseteq\W_k$.

	\begin{lemma}\label{lem:directed-unions}
		The union of a nonempty family of wide subcategories that is
		directed under inclusion is wide.
	\end{lemma}
	
	\begin{proof}
		Let $\W=\bigcup_{i\in I}\W_i$ be such a union, viewed as a full
		subcategory. It contains zero and is closed under isomorphisms and
		direct summands. Any two objects lie in one $\W_i$, so their
		direct sum also lies in $\W$. If $f:X\to Y$ is a morphism in
		$\W$, choose $i$ with $X,Y\in\W_i$. An $n$-kernel and an
		$n$-cokernel supplied by $\W_i$ have all their terms in $\W$.
		Likewise, an $n$-extension with endpoints in $\W$ has both
		endpoints in some $\W_i$, where it has an equivalent representative
		with all middle terms in $\W_i$. Thus $\W$ is wide.
	\end{proof}
	
	\section{Spectrum and support theory}\label{sec:spectrum}
	
	From now on, $\M$ is an essentially small Krull--Schmidt $n$-abelian category.
	Fix an arbitrary subset $\mathscr W\subseteq\Wide_n(\M)$.
	No closure property of $\mathscr W$ is assumed. 
	For a class $\mathcal S$ of objects of $\M$, define
	$\mathbf{Z}(\mathcal S)=\{\Pp\in\mathscr W\mid \Pp\cap\mathcal S=\varnothing\}$.
	For classes $\mathcal S_1$ and $\mathcal S_2$, set
	$\mathcal S_1\oplus\mathcal S_2
	=\{M_1\oplus M_2\mid M_i\in\mathcal S_i\}$.
	
	\begin{proposition}\label{prop:topology}
		The subsets $\mathbf{Z}(\mathcal S)$ are the closed sets of a topology on $\mathscr W$.
	\end{proposition}
	
	\begin{proof}
		For any family $\{\mathcal S_\lambda\}$ of classes of objects of $\M$, we have
		\[
		\bigcap_\lambda \mathbf{Z}(\mathcal S_\lambda)
		=
		\mathbf{Z}\!\left(\bigcup_\lambda\mathcal S_\lambda\right).
		\]
		For the union of two closed sets, let $\Pp\in\mathscr W$. Since
		$\Pp$ is wide, it is closed under finite direct sums and direct
		summands. Hence
		$\Pp\cap(\mathcal S_1\oplus\mathcal S_2)\neq\varnothing$
		if and only if $\Pp\cap\mathcal S_1\neq\varnothing$ and
		$\Pp\cap\mathcal S_2\neq\varnothing$.
		Indeed, if $M_1\oplus M_2\in\Pp$, then closure under direct
		summands gives $M_1,M_2\in\Pp$, while the converse follows from
		closure under finite direct sums.
		Taking complements in $\mathscr W$ yields
		$\mathbf{Z}(\mathcal S_1)\cup \mathbf{Z}(\mathcal S_2)
		=
		\mathbf{Z}(\mathcal S_1\oplus\mathcal S_2)$.
		Finally,
		$\mathbf{Z}(\varnothing)=\mathscr W$, while
		$\mathbf{Z}(\Ob\M)=\varnothing$, since every wide subcategory
		contains the zero object.
	\end{proof}
	
	\begin{definition}
		The resulting topological space is denoted by
		$\Spec_{\mathscr W}(\M)$ and is called the \emph{spectrum of $\M$
			with respect to $\mathscr W$}. For an object $M\in\M$, define its
		\emph{support} by
		$\Supp(M)=\mathbf{Z}(\{M\})=\{\Pp\in\mathscr W\mid M\notin\Pp\}$,
		and set
		$\mathbf{U}(M)=\{\Pp\in\mathscr W\mid M\in\Pp\}$.
	\end{definition}
	
	Every closed subset of $\Spec_{\mathscr W}(\M)$ is of the form
	\begin{equation}\label{eq:closed-intersection-support}
		\mathbf{Z}(\mathcal S)=\bigcap_{M\in\mathcal S}\Supp(M),
	\end{equation}
	for some class $\mathcal S$ of objects of $\M$. The subsets $\mathbf{U}(M)$ form
	an open basis, since
	\[
	\mathbf{U}(M)\cap \mathbf{U}(N)=\mathbf{U}(M\oplus N).
	\]
	Moreover, the support satisfies
	$\Supp(0)=\varnothing,$ and 
	$\Supp(M\oplus N)=\Supp(M)\cup\Supp(N)$.

	\begin{proposition}\label{prop:point-closure}
		For $\Pp\in\Spec_{\mathscr W}(\M)$,
		\[
		\overline{\{\Pp\}}
		=\{\Qq\in\mathscr W\mid \Qq\subseteq\Pp\}.
		\]
		In particular, $\Spec_{\mathscr W}(\M)$ is a $T_0$-space.
	\end{proposition}
	
	\begin{proof}
		Let $\mathcal S_{\Pp}=\Ob\M\setminus\Ob\Pp$. By definition,
		$\mathbf{Z}(\mathcal S_{\Pp})=\{\Qq\in\mathscr W\mid \Qq\subseteq\Pp\}$,
		which is a closed subset containing $\Pp$.
		Now let $\mathbf{Z}(\mathcal S)$ be a closed subset containing $\Pp$. Then
		$\Pp\cap\mathcal S=\varnothing$, so $\mathcal S\subseteq\mathcal S_{\Pp}$.
		Hence $\mathbf{Z}(\mathcal S_{\Pp})\subseteq \mathbf{Z}(\mathcal S)$. Therefore
		$\mathbf{Z}(\mathcal S_{\Pp})$ is the smallest closed subset containing $\Pp$,
		and thus equals $\overline{\{\Pp\}}$.
		Finally, if two points have the same closure, then the above description
		implies that the corresponding wide subcategories contain each other.
		Hence they are equal, and $\Spec_{\mathscr W}(\M)$ is a $T_0$-space.
	\end{proof}
	
	Thus the specialization order on $\Spec_{\mathscr W}(\M)$ is
	the inclusion order of its wide subcategories.
	For the canonical support, define
	$\Supp^{-1}(\omega)=\{M\in\M\mid\Supp(M)\subseteq\omega\}$.
	
	\begin{lemma}\label{lem:canonical-inverse}
		For every subset $\omega\subseteq\Spec_{\mathscr W}(\M)$,
		\begin{equation}\label{eq:inverse-intersection}
			\Supp^{-1}(\omega)
			=\bigcap_{\Pp\notin\omega}\Pp.
		\end{equation}
		In particular, $\Supp^{-1}(\omega)$ is wide.
	\end{lemma}
	
	\begin{proof}
		For $M\in\M$, we have
		\[
		\begin{aligned}
			M\in\Supp^{-1}(\omega)
			&\Longleftrightarrow \Supp(M)\subseteq\omega\\
			&\Longleftrightarrow \Pp\notin\Supp(M)\text{ for all }\Pp\notin\omega\\
			&\Longleftrightarrow M\in\Pp\text{ for all }\Pp\notin\omega,
			\qquad\text{by }\Supp(M)=\{\Pp\mid M\notin\Pp\}.
		\end{aligned}
		\]
		This proves \eqref{eq:inverse-intersection}. Since every $\Pp$ is a
		wide subcategory, Proposition~\ref{prop:intersection} implies that
		$\Supp^{-1}(\omega)$ is wide. If $\omega=\Spec_{\mathscr W}(\M)$,
		the intersection on the right is empty and hence equals $\M$, which is
		wide.
	\end{proof}
	
	For a subcategory $\X\subseteq\M$ define $\Supp(\X)=\bigcup_{M\in\X}\Supp(M)$. Then
	\begin{equation}\label{eq:subcategory-support}
		\Supp(\X)
		=\{\Pp\in\mathscr W\mid \X\nsubseteq\Pp\}.
	\end{equation}
	
	\begin{lemma}\label{prop:subcat-spcl}
		For every subcategory $\X\subseteq\M$, the subset $\Supp(\X)$ is
		specialization-closed.
	\end{lemma}
	
	\begin{proof}
		Let $\Pp\in\Supp(\X)$ and let
		$\Qq\in\overline{\{\Pp\}}$. By
		Proposition~\ref{prop:point-closure}, we have $\Qq\subseteq\Pp$.
		Since $\Pp\in\Supp(\X)$, there exists some $M\in\X$ with
		$M\notin\Pp$. Hence $M\notin\Qq$, and therefore
		$\Qq\in\Supp(\X)$.
	\end{proof}
	
	For a subset $\omega\subseteq\Spec_{\mathscr W}(\M)$, its closure is
	given by
	\[
	\overline{\omega}
	=\bigcap_{M\in\M,\,\omega\subseteq\Supp(M)}\Supp(M).
	\]
	Indeed, every closed subset is an intersection of supports of objects of
	$\M$ by \eqref{eq:closed-intersection-support}.
	
	The inverse-support formula \eqref{eq:inverse-intersection} expresses
	support containment as an intersection of wide subcategories. We next
	characterize the assignments for which every support-containment
	subcategory is wide.
	Let $X$ be a topological space and let $\Ccl(X)$ denote the collection
	of closed subsets of $X$.
	
	\begin{definition}\label{def:n-support-datum}
		An \emph{$n$-support datum} on $\M$ is a pair $(X,\sigma)$, where
		\[
		\sigma:\Ob\M\longrightarrow\Ccl(X)
		\]
		is constant on isomorphism classes and satisfies the following conditions.
		\begin{enumerate}[label=\textup{(NSD\arabic*)},leftmargin=*]
			\item $\sigma(0)=\varnothing$. \label{NSD1}
			\item $\sigma(M\oplus N)=\sigma(M)\cup\sigma(N)$.\label{NSD2}
			\item If $K_n\to\cdots\to K_1\to M$
			is a minimal $n$-kernel of a morphism $f:M\to N$, then
			\[
			\sigma(K_i)\subseteq\sigma(M)\cup\sigma(N),
			\quad 1\leq i\leq n.
			\]\label{NSD3}
			\item If $N\to C_1\to\cdots\to C_n$ is a minimal $n$-cokernel of $f:M\to N$, then
			\[
			\sigma(C_i)\subseteq\sigma(M)\cup\sigma(N),
			\quad 1\leq i\leq n.
			\]\label{NSD4}
			\item If $X_0\to E_n\to\cdots\to E_1\to Y_0$ is a minimal $n$-extension of $Y_0$ by $X_0$, then
			\[
			\sigma(E_i)\subseteq\sigma(X_0)\cup\sigma(Y_0),
			\quad 1\leq i\leq n.
			\]\label{NSD5}
		\end{enumerate}
	\end{definition}
	
	\begin{remark}
		For $n\geq2$, the minimality assumptions in
		Definition~\ref{def:n-support-datum} cannot be omitted. Indeed, for any
		object $C$, the sequence
		$0\longrightarrow C\xrightarrow{1_C}C$
		is a $2$-cokernel of the zero morphism $0\to0$, see
		\cite[Remark~2.3]{J16}. Taking $\mathscr W=\{\mathbf0\}$ and
		$C\neq0$ would then give the false inclusion
		$\Supp(C)\subseteq\Supp(0)$. For larger $n$, the same example can be padded with zero terms.
		Thus conditions imposed on all representatives would generally fail;
		minimal representatives remove this obstruction.
	\end{remark}
	
	For a subset $\omega\subseteq X$, define the subcategory
	\[
	\sigma^{-1}(\omega)
	=\{M\in\M\mid\sigma(M)\subseteq\omega\}.
	\]
	The support axioms first give a sufficient condition for these
	subcategories to be wide.
	
	\begin{lemma}\label{lem:support-preimages-wide}
		If $(X,\sigma)$ is an $n$-support datum, then
		$\sigma^{-1}(\omega)$ is wide for every subset $\omega\subseteq X$.
	\end{lemma}
	
	\begin{proof}
		Fix $\omega\subseteq X$. By \ref{NSD1} and \ref{NSD2}, the subcategory
		$\sigma^{-1}(\omega)$ contains the zero object and is closed under
		finite direct sums. It is closed under isomorphisms by assumption.
		If $M\oplus N\in\sigma^{-1}(\omega)$, then
		\[
		\sigma(M)\cup\sigma(N)=\sigma(M\oplus N)\subseteq\omega,
		\]
		so both summands belong to $\sigma^{-1}(\omega)$.
		
		For a morphism $f:M\to N$ in this subcategory, the union
		$\sigma(M)\cup\sigma(N)$ is contained in $\omega$.
		Conditions \ref{NSD3} and \ref{NSD4} therefore put every term of
		its minimal $n$-kernel and minimal $n$-cokernel in
		$\sigma^{-1}(\omega)$. If the endpoints of a minimal $n$-extension
		belong to this subcategory, condition \ref{NSD5} puts all its middle
		terms there as well. Lemma~\ref{lem:minimal-characterization} now
		shows that $\sigma^{-1}(\omega)$ is wide.
	\end{proof}
	
	Conversely, requiring all support-containment subcategories to be wide
	recovers each support axiom. In particular, the subsets used to test this
	condition need not be closed or specialization-closed.
	
	\begin{theorem}\label{prop:support-characterization}
		Let \(X\) be a topological space and let
		$\sigma:\Ob\M\longrightarrow\Ccl(X)$
		be an assignment that is constant on isomorphism classes.
		Then the following conditions are equivalent:
		\begin{enumerate}
			\item \((X,\sigma)\) is an \(n\)-support datum.
			
			\item For every subset \(\omega\subseteq X\), the full subcategory
			\[
			\sigma^{-1}(\omega)
			=
			\{M\in\mathcal M\mid \sigma(M)\subseteq\omega\}
			\]
			is wide.
		\end{enumerate}
	\end{theorem}
	
	\begin{proof}
		The forward implication is Lemma~\ref{lem:support-preimages-wide}.
		For the converse, assume that $\sigma^{-1}(\omega)$ is wide for
		every $\omega\subseteq X$.
		
		Taking $\omega=\varnothing$, we obtain
		$0\in\sigma^{-1}(\varnothing)$, and hence
		$\sigma(0)=\varnothing$. This proves \ref{NSD1}.
		
		For $M,N\in\M$, take $\omega=\sigma(M)\cup\sigma(N)$.
		The objects $M$ and $N$ belong to the wide subcategory
		$\sigma^{-1}(\omega)$, so their direct sum belongs to it. Thus
		\[
		\sigma(M\oplus N)\subseteq\sigma(M)\cup\sigma(N).
		\]
		Taking instead $\omega=\sigma(M\oplus N)$ and using closure under
		direct summands gives the reverse inclusion. This proves \ref{NSD2}.
		
		Let $f:M\to N$ be a morphism and again put
		$\omega=\sigma(M)\cup\sigma(N)$. The morphism lies in the full
		wide subcategory $\sigma^{-1}(\omega)$. By
		Lemma~\ref{lem:minimal-characterization}, every term of its minimal
		$n$-kernel and minimal $n$-cokernel belongs to
		$\sigma^{-1}(\omega)$. These are exactly the inclusions in
		\ref{NSD3} and \ref{NSD4}.
		
		Finally, for endpoints $X_0,Y_0\in\M$, put
		$\omega=\sigma(X_0)\cup\sigma(Y_0)$. Applying
		Lemma~\ref{lem:minimal-characterization} to the wide subcategory
		$\sigma^{-1}(\omega)$ puts every middle term of every minimal
		$n$-extension of $Y_0$ by $X_0$ in this subcategory. This is
		\ref{NSD5}, completing the converse.
	\end{proof}
	
	For the canonical assignment, Lemma~\ref{lem:canonical-inverse}
	already supplies the wide subcategories required by
	Theorem~\ref{prop:support-characterization}.
	
	\begin{corollary}\label{cor:canonical-support}
		For every family $\mathscr W\subseteq\Wide_n(\M)$, the assignment
		\[
		\Supp:\Ob\M\longrightarrow\Ccl(\Spec_{\mathscr W}(\M))
		\]
		defines an $n$-support datum.
	\end{corollary}
	
	\begin{proof}
		The assignment is constant on isomorphism classes because each member
		of $\mathscr W$ is closed under isomorphisms. Each $\Supp(M)$ is
		closed because $\Supp(M)=\mathbf Z(\{M\})$. By
		Lemma~\ref{lem:canonical-inverse}, $\Supp^{-1}(\omega)$ is wide for
		every subset $\omega$ of the spectrum. The conclusion follows from
		Theorem~\ref{prop:support-characterization}.
	\end{proof}
	
	\begin{remark}
		For $n=1$, the equivalence in
		Theorem~\ref{prop:support-characterization} holds in any
		essentially small abelian category, without the Krull--Schmidt
		hypothesis.
	\end{remark}
	
	\section{Radical wide subcategories}\label{sec:radical}
	
	The support correspondence determines a closure operator on
	$\Wide_n(\M)$. We describe its fixed points and distinguish the
	subsets realized by supports from all specialization-closed subsets.
	Keep the family $\mathscr W\subseteq\Wide_n(\M)$ fixed.
	
	\begin{definition}\label{def:radical}
		For $\X\in\Wide_n(\M)$, its $\mathscr W$-\emph{radical} is
		\begin{equation}\label{eq:radical}
			\sqrt[\mathscr W]{\X}
			=\bigcap_{\Pp\in\mathscr W,~~\X\subseteq\Pp}\Pp.
		\end{equation}
		The empty intersection is $\M$. We call $\X$
		$\mathscr W$-\emph{radical} if $\sqrt[\mathscr W]{\X}=\X$, and
		write $\Rad_{\mathscr W}(\M)$ for the poset of such subcategories.
	\end{definition}
	
	The subcategory in \eqref{eq:radical} is wide by Proposition~\ref{prop:intersection}.
	
	\begin{remark}\label{rem:two-radicals}
		The notation $\rad_{\M}(U,V)$ denotes a group of morphisms in the
		categorical Jacobson radical. By contrast,
		$\sqrt[\mathscr W]{\X}$ is a subcategory obtained by intersecting
		members of the chosen family $\mathscr W$. The former is used to
		define minimal complexes; the latter is a closure operation on
		$\Wide_n(\M)$ and depends on $\mathscr W$.
	\end{remark}
	
	The members of $\mathscr W$ containing a subcategory also contain
	its radical. This observation gives the closure properties.
	
	\begin{lemma}\label{lem:radical-containment}
		For every wide subcategory $\X$, we have
		$\X\subseteq\sqrt[\mathscr W]{\X}$. Moreover, for every
		$\Pp\in\mathscr W$,
		\begin{equation}\label{eq:radical-containment}
			\X\subseteq\Pp
			\quad\Longleftrightarrow\quad
			\sqrt[\mathscr W]{\X}\subseteq\Pp.
		\end{equation}
	\end{lemma}
	
	\begin{proof}
		Every factor in \eqref{eq:radical} contains $\X$, so their
		intersection contains $\X$. This also holds for the empty intersection.
		If $\X\subseteq\Pp\in\mathscr W$, then $\Pp$ is one of those
		factors, giving $\sqrt[\mathscr W]{\X}\subseteq\Pp$.
		The reverse implication follows from
		$\X\subseteq\sqrt[\mathscr W]{\X}$.
	\end{proof}
	
	\begin{proposition}\label{prop:radical-closure}
		The operation $\sqrt[\mathscr W]{-}$ is a closure operator on
		$\Wide_n(\M)$: it is extensive, monotone, and idempotent.
	\end{proposition}
	
	\begin{proof}
		The extensivity follows from Lemma~\ref{lem:radical-containment}.
		If $\X\subseteq\mathcal Y$, every member of $\mathscr W$
		containing $\mathcal Y$ contains $\X$. Intersecting these families
		gives $\sqrt[\mathscr W]{\X}\subseteq\sqrt[\mathscr W]{\mathcal Y}$.
		Finally, by \eqref{eq:radical-containment}, the same members of
		\(\mathscr W\) contain \(\X\) and
		\(\sqrt[\mathscr W]{\X}\). Their defining intersections are therefore equal,
		which proves idempotence.
	\end{proof}
	
	\begin{corollary}\label{cor:radical-intersections}
		The $\mathscr W$-radical wide subcategories are exactly the
		intersections of subfamilies of $\mathscr W$, including the empty
		intersection $\M$.
	\end{corollary}
	
	\begin{proof}
		If $\X$ is radical, equation~\eqref{eq:radical} gives the required
		expression. Conversely, suppose that
		$\X=\bigcap_{\Pp\in\mathscr V}\Pp$ for a subfamily
		$\mathscr V\subseteq\mathscr W$. Every member of $\mathscr V$
		contains $\X$, so the intersection defining its radical is contained
		in $\X$. Extensivity gives
		$\X\subseteq\sqrt[\mathscr W]{\X}\subseteq\X$.
	\end{proof}
	
	The inverse-support formula \eqref{eq:inverse-intersection} identifies this closure operator with
	the composite of the two support assignments.
	
	\begin{lemma}\label{lem:radical-support}
		For every wide subcategory $\X$,
		\begin{equation}\label{eq:radical-support-composite}
			\Supp^{-1}(\Supp(\X))=\sqrt[\mathscr W]{\X},
			\qquad
			\Supp(\sqrt[\mathscr W]{\X})=\Supp(\X).
		\end{equation}
	\end{lemma}
	
	\begin{proof}
		By \eqref{eq:inverse-intersection} and
		\eqref{eq:subcategory-support},
		\[
		\Supp^{-1}(\Supp(\X))
		=\bigcap_{\Pp\notin\Supp(\X)}\Pp
		=\bigcap_{\Pp\in\mathscr W,~~\X\subseteq\Pp}\Pp
		=\sqrt[\mathscr W]{\X}.
		\]
		For each $\Pp\in\mathscr W$, equation~\eqref{eq:radical-containment}
		gives $\X\nsubseteq\Pp$ if and only if
		$\sqrt[\mathscr W]{\X}\nsubseteq\Pp$. Formula
		\eqref{eq:subcategory-support} therefore gives equality of the two
		supports.
	\end{proof}
	
	Define the poset of realized supports by
	\[
	\Param_{\mathscr W}(\M)
	=\{\Supp(\X)\mid\X\in\Wide_n(\M)\},
	\]
	ordered by inclusion. Every member is specialization-closed, but
	there need not be an equality with all specialization-closed
	subsets of the spectrum.
	
	\begin{theorem}\label{thm:rad-param}
		The assignments \(\Supp\) and \(\Supp^{-1}\) induce mutually inverse
		order isomorphisms
		\[
		\Rad_{\mathscr W}(\M)
		\overset{\Supp}{\underset{\Supp^{-1}}{\rightleftarrows}}
		\Param_{\mathscr W}(\M).
		\]
	\end{theorem}
	
	\begin{proof}
		For a radical $\X$, \eqref{eq:radical-support-composite} gives
		$\Supp^{-1}(\Supp(\X))=\X$. For
		$\omega=\Supp(\mathcal Y)$, Proposition~\ref{prop:radical-closure}
		and Lemma~\ref{lem:radical-support} show that
		$\Supp^{-1}(\omega)=\sqrt[\mathscr W]{\mathcal Y}$ is radical
		and has support $\omega$. Both assignments preserve inclusion
		by their definitions.
	\end{proof}
	
	\begin{remark}\label{rem:galois}
		For a wide subcategory $\X$ and any subset $\omega$ of the spectrum,
		\[
		\Supp(\X)\subseteq\omega
		\quad\Longleftrightarrow\quad
		\X\subseteq\Supp^{-1}(\omega).
		\]
		Thus $\Supp\dashv\Supp^{-1}$ is an adjunction of posets, with
		closure operator $\Supp^{-1}\circ\Supp=\sqrt[\mathscr W]{-}$.
		Theorem~\ref{thm:rad-param} describes its fixed points.
	\end{remark}
	
	\section{Prime wide subcategories and reconstruction}\label{sec:prime}
	
	We prove prime separation and reconstruct classifying support
	spaces from the prime spectrum.
	From now on, $\Supp$ denotes the canonical support on
	$\Spec_{\Prime_n(\M)}(\M)$.
	
	\begin{definition}\label{def:prime}
		A proper wide subcategory $\Pp\subsetneq\M$ is \emph{prime} if 
		\[
		\{\X\in\Wide_n(\M)\mid \Pp\subsetneq\X\}
		\]
		has a least element.  We write $\Prime_n(\M)$ for the set of prime wide subcategories.
	\end{definition}
	
	\begin{remark}
		We follow the least-element convention in
		\cite[Remark~2.9]{M23}, since a unique minimal element
		need not be a least element in an infinite poset.
		This notion of primeness need not coincide with
		meet-primality in the lattice.
	\end{remark}
	
	Since $\Wide_n(\M)$ is complete, for every proper $\Pp$ we may define
	\[
	\Pp^{\dagger}
	=\bigcap_{\Pp\subsetneq\X}\X.
	\]
	The indexing family is nonempty because $\Pp\subsetneq\M$.
	
	\begin{lemma}\label{lem:prime-dagger}
		A proper wide subcategory $\Pp$ is prime if and only if $\Pp\subsetneq\Pp^{\dagger}$.
		In this case $\Pp^{\dagger}$ is the least wide subcategory properly containing $\Pp$.
	\end{lemma}
	
	\begin{proof}
		If $\Pp^+$ is the least wide subcategory properly containing $\Pp$, then $\Pp^+$ is contained in every wide subcategory properly containing $\Pp$ and also belongs to this family, hence $\Pp^{\dagger}=\Pp^+$.  Conversely, if $\Pp\subsetneq\Pp^{\dagger}$, then by construction $\Pp^{\dagger}$ is contained in every wide subcategory properly containing $\Pp$.
	\end{proof}
	
	Recall that an element \(p\) of a complete lattice \(L\) is
	\emph{completely meet-irreducible} if, whenever \(p\) is the meet of a
	family of elements of \(L\), that family contains \(p\). The empty meet
	is allowed. Consequently, the top element is not completely
	meet-irreducible.
	
	\begin{proposition}\label{prop:prime-cmi}
		A wide subcategory $\Pp$ is prime if and only if it is completely meet-irreducible in $\Wide_n(\M)$.
	\end{proposition}
	
	\begin{proof}
		Suppose $\Pp$ is prime, and let $\Pp^+$ be the least wide subcategory properly containing $\Pp$.  If
		$\Pp=\bigcap_{i\in I}\X_i$
		and no $\X_i$ equals $\Pp$, then every $\X_i$ properly contains $\Pp$, so $\Pp^+\subseteq\X_i$ for all $i$.  Hence $\Pp^+\subseteq\bigcap_i\X_i=\Pp$, a contradiction.
		
		Conversely, suppose $\Pp$ is completely meet-irreducible.  It is not the top element.  If $\Pp^{\dagger}=\Pp$, then $\Pp$ is the meet of the family of wide subcategories properly containing it.  Complete meet-irreducibility would force one member of that family to equal $\Pp$, a contradiction.  Thus $\Pp\subsetneq\Pp^{\dagger}$ and Lemma~\ref{lem:prime-dagger} applies.
	\end{proof}
	
	Directed unions provide enough prime wide subcategories to separate wide subcategories
	from objects. This gives a stronger conclusion for the full family
	of prime wide subcategories than for an arbitrary chosen family $\mathscr W$.
	
	\begin{proposition}\label{prop:prime-separation}
		Let $\X$ be a wide subcategory of $\M$ and let $M\notin\X$.
		There is a prime wide subcategory $\Pp$ with
		$\X\subseteq\Pp$ and $M\notin\Pp$. Consequently,
		\begin{equation}\label{eq:prime-intersection}
			\X=\bigcap_{\substack{\Pp\in\Prime_n(\M)\\\X\subseteq\Pp}}\Pp,
			\qquad
			\Rad_{\Prime_n(\M)}(\M)=\Wide_n(\M).
		\end{equation}
	\end{proposition}
	
	\begin{proof}
		The collection
		\[
		\mathsf S=\{\W\in\Wide_n(\M)\mid
		\X\subseteq\W,\ M\notin\W\}
		\]
		is a set because $\M$ is essentially small, and it contains $\X$.
		Every nonempty chain in $\mathsf S$ has its union as an upper
		bound in $\mathsf S$, by Lemma~\ref{lem:directed-unions}.
		The empty chain has the upper bound $\X$. By Zorn's lemma, there exists a maximal element \(\Pp\) of
		\(\mathsf S\).
		Every wide subcategory strictly containing $\Pp$ must contain $M$.
		Hence
		\[
		\Pp^+=\wide_{\M}(\Ob\Pp\cup\{M\})
		\]
		is the least wide subcategory strictly containing $\Pp$.
		Thus $\Pp$ is prime and has the required separation property.
		Separating each object outside $\X$ proves the first equality in
		\eqref{eq:prime-intersection}. For $\X=\M$, this equality uses the
		empty-intersection convention. The second equality in
		\eqref{eq:prime-intersection} then follows from Definition~\ref{def:radical}.
	\end{proof}
	
	In particular, the canonical prime support separates all wide
	subcategories. It gives a classification by specialization-closed
	subsets precisely when every such subset is realized as the support
	of a wide subcategory. To obtain a classifying support datum in the
	sense of Definition~\ref{def:classifying}, the prime spectrum must
	also be noetherian and sober.
	
	We next give the corresponding topological description. Let $X$ be a topological space.  Write $\Spcl(X)$ for the complete lattice of specialization-closed subsets.  Define
	\[
	\Spcl_1(X)
	=\left\{V\in\Spcl(X)\ \middle|\ 
	\{T\in\Spcl(X)\mid V\subsetneq T\}
	\text{ has a least element}\right\}.
	\]
	For $x\in X$, put
	$V_x=\{y\in X\mid x\notin\overline{\{y\}}\}.$

	\begin{lemma}\label{lem:spcl1-points}
		If $X$ is $T_0$, the map
		\[
		X\longrightarrow\Spcl_1(X),
		\qquad
		x\longmapsto V_x,
		\]
		is a bijection. The least specialization-closed subset properly containing $V_x$ is $V_x\cup\{x\}$.
	\end{lemma}
	
	\begin{proof}
		First $V_x$ is specialization-closed.  Indeed, if $y\in V_x$ and $z\in\overline{\{y\}}$, then $x\in\overline{\{z\}}$ would imply $x\in\overline{\{y\}}$ by transitivity of specialization, a contradiction.
		
		We next show that $V_x\cup\{x\}$ is specialization-closed.  If $z\in\overline{\{x\}}$ and $z\notin V_x$, then $x\in\overline{\{z\}}$.  Since $X$ is $T_0$, mutual specialization forces $z=x$.  Now let $T$ be specialization-closed with $V_x\subsetneq T$.  Choose $y\in T\setminus V_x$.  Then $x\in\overline{\{y\}}\subseteq T$, hence $V_x\cup\{x\}\subseteq T$.  Therefore $V_x\in\Spcl_1(X)$.
		
		Conversely, let $V\in\Spcl_1(X)$ and let $T$ be the least specialization-closed subset properly containing $V$.  Choose $x\in T\setminus V$.  If $y\in V$, then $x\notin\overline{\{y\}}$, since otherwise specialization-closedness of $V$ would imply $x\in V$.  Hence $V\subseteq V_x$.
		
		Suppose that some $y\in V_x\setminus V$.  Then
		$V\cup\overline{\{y\}}$
		is a specialization-closed subset properly containing $V$.  Since $T$ is the least such subset, $T\subseteq V\cup\overline{\{y\}}$.
		Since $x\in T\setminus V$, it follows that $x\in\overline{\{y\}}$, contradicting $y\in V_x$.  Hence $V=V_x$.
		
		Finally, if $V_x=V_y$, then $x\notin V_y$ and $y\notin V_x$, so $y\in\overline{\{x\}}$ and $x\in\overline{\{y\}}$.  The $T_0$ property gives $x=y$.
	\end{proof}
	
	By the same complete-lattice argument as in Lemma~\ref{lem:prime-dagger} and Proposition~\ref{prop:prime-cmi}, the elements of $\Spcl_1(X)$ are precisely the completely meet-irreducible elements of the complete lattice $\Spcl(X)$.
	
	We apply the point correspondence in Lemma~\ref{lem:spcl1-points}
	to support data that classify wide subcategories by specialization-closed
	subsets.
	
	\begin{definition}\label{def:classifying}
		An $n$-support datum $(X,\sigma)$ on $\M$ is \emph{classifying} if:
		\begin{enumerate}[label=\textup{(CNSD\arabic*)},leftmargin=*]
			\item $X$ is noetherian and sober; 
			\item the assignments
			\[
			\sigma(\W)=\bigcup_{M\in\W}\sigma(M),
			\qquad
			\sigma^{-1}(V)=\{M\in\M\mid\sigma(M)\subseteq V\}
			\]
			induce mutually inverse lattice isomorphisms
			\begin{equation}\label{eq:classifying-lattice}
				\Wide_n(\M)
				\overset{\sigma}{\underset{\sigma^{-1}}{\rightleftarrows}}
				\Spcl(X).
			\end{equation}
		\end{enumerate}
	\end{definition}
	
	A sober space is $T_0$, so Lemma~\ref{lem:spcl1-points} applies.  Since lattice isomorphisms preserve completely meet-irreducible elements, Proposition~\ref{prop:prime-cmi} and Lemma~\ref{lem:spcl1-points} already yield a canonical bijection
	\[
	X\longrightarrow\Prime_n(\M).
	\]
	The next theorem gives the objectwise description of this bijection
	and proves that it is a homeomorphism.
	
	\begin{theorem}\label{thm:reconstruction}
		Let $(X,\sigma)$ be a classifying $n$-support datum on $\M$.
		Define
		\begin{equation}\label{eq:phi}
			\phi:X\longrightarrow\Spec_{\Prime_n(\M)}(\M),
			\qquad \phi(x)=\sigma^{-1}(V_x).
		\end{equation}
		Then $\phi$ is a homeomorphism. Moreover,
		\begin{equation}\label{eq:phi-object}
			\phi(x)=\{M\in\M\mid x\notin\sigma(M)\}.
		\end{equation}
		For every $M\in\M$, this homeomorphism satisfies
		$\phi^{-1}(\Supp(M))=\sigma(M)$.
	\end{theorem}
	
	\begin{proof}
		We first identify the points. Sobriety implies that $X$ is $T_0$,
		so Lemma~\ref{lem:spcl1-points} gives a bijection
		$x\mapsto V_x$ from $X$ to $\Spcl_1(X)$.
		Since the lattice isomorphism preserves arbitrary meets, it preserves completely meet-irreducible elements. Thus it
		identifies $\Spcl_1(X)$ with the completely meet-irreducible
		elements of $\Wide_n(\M)$, which are the prime wide subcategories
		by Proposition~\ref{prop:prime-cmi}. This proves that $\phi$ is
		well defined and bijective.
		
		For each $M\in\M$, closedness of $\sigma(M)$ gives
		\[
		\sigma(M)\subseteq V_x
		\quad\Longleftrightarrow\quad x\notin\sigma(M).
		\]
		Indeed, $x\notin V_x$, which proves the forward implication.
		For the reverse implication, if $x\notin\sigma(M)$ and
		$y\in\sigma(M)$, then
		$\overline{\{y\}}\subseteq\sigma(M)$. Hence
		$x\notin\overline{\{y\}}$, so $y\in V_x$.
		This establishes \eqref{eq:phi-object} and therefore
		\begin{equation}\label{eq:phi-preimage-support}
			\phi^{-1}(\Supp(M))=\sigma(M).
		\end{equation}
		Every closed subset of the prime spectrum is an intersection of
		object supports, by \eqref{eq:closed-intersection-support}.
		Its inverse image under $\phi$ is consequently an intersection
		of closed subsets of $X$. Thus $\phi$ is continuous.
		
		It remains to prove that $\phi$ is closed. We show first that
		every closed subset of $X$ is the support of one object.
		Let $C\subseteq X$ be nonempty and closed.
		Noetherianity means that descending chains of closed subsets
		stabilize. It follows that $C$ is a finite union of nonempty
		irreducible closed subsets: otherwise choose a minimal closed
		subset $F$ that has no such decomposition. The set $F$ cannot
		be irreducible, so it is the union of two proper closed subsets
		of $F$, which are also closed in $X$. Each has a finite
		irreducible decomposition by minimality, giving one for $F$, a
		contradiction. Write the resulting decomposition as
		$C=C_1\cup\cdots\cup C_r$.
		Since $X$ is sober, each $C_i$ has a unique generic point
		$x_i$, that is, $C_i=\overline{\{x_i\}}$. Hence
		$C=\overline{\{x_1\}}\cup\cdots\cup\overline{\{x_r\}}$.

		For each $i$, the point closure $\overline{\{x_i\}}$ is
		specialization-closed. The two assignments in
		\eqref{eq:classifying-lattice} are inverse, so
		\[
		\overline{\{x_i\}}
		=\bigcup_{M\in\sigma^{-1}(\overline{\{x_i\}})}\sigma(M).
		\]
		As $x_i$ belongs to this union, there is an object $M_i$ with
		$x_i\in\sigma(M_i)\subseteq\overline{\{x_i\}}$.
		
		The set $\sigma(M_i)$ is closed and contains $x_i$, so it also
		contains $\overline{\{x_i\}}$. Therefore
		$\sigma(M_i)=\overline{\{x_i\}}$.
		The direct-sum axiom yields
		$C=\sigma(M_1\oplus\cdots\oplus M_r)$. For $C=\varnothing$, the same assertion holds with the zero
		object, by $\sigma(0)=\varnothing$.
		
		Thus, for every closed $C\subseteq X$, there is an object
		$M_C$ with $C=\sigma(M_C)$. Equation
		\eqref{eq:phi-preimage-support} and surjectivity of $\phi$ give 
		$\phi(C)
		=\phi\bigl(\phi^{-1}(\Supp(M_C))\bigr)
		=\Supp(M_C)$, 
		which is closed in the prime spectrum. Hence $\phi$ is a closed
		continuous bijection, and therefore a homeomorphism.
	\end{proof}
	
	\begin{remark}\label{rem:weaker-topology}
		The conclusion of Theorem~\ref{thm:reconstruction} remains valid if the
		topological assumptions in Definition~\ref{def:classifying} are replaced
		by the following weaker conditions: \(X\) is \(T_0\), and every nonempty
		closed subset of \(X\) is a finite union of closures of points. All other
		assumptions, including the \(n\)-support axioms and the two inverse maps
		in \eqref{eq:classifying-lattice}, remain unchanged.
		
		The \(T_0\) condition yields the bijection in
		Lemma~\ref{lem:spcl1-points}. The finite-union condition implies that
		every closed subset is realized as the support of an object, as in the
		proof of Theorem~\ref{thm:reconstruction}. Equation
		\eqref{eq:phi-preimage-support} then shows that \(\phi\) is closed.
	\end{remark}
	
	\section{A \texorpdfstring{$2$}{2}-abelian example with non-classifying prime support}\label{sec:example}
	
	We conclude with a concrete example in which all wide subcategories, prime wide subcategories, supports, and radicals can be computed directly from a bound quiver. The algebra and its higher wide subcategories occur in \cite[Section~7]{HJV20} and \cite[Section~7.1]{F26}. We include the module calculations to make the subsequent spectrum and support computations explicit.
	
	Let $\K$ be a field and
	\begin{equation}\label{eq:example-algebra}
		\Lambda
		=\K Q/\langle\beta\alpha\rangle,
		\qquad
		Q:\quad
		1\xrightarrow{\alpha}2\xrightarrow{\beta}3.
	\end{equation}
	We compose paths from right to left, so $\beta\alpha$ means first
	$\alpha$ and then $\beta$. We work with finite-dimensional left
	$\Lambda$-modules, identified with representations of $Q$ in its
	displayed orientation; thus $P_i=\Lambda e_i$. Put
	\[
	A=S_1,
	\qquad
	B=P_1,
	\qquad
	C=P_2,
	\qquad
	D=P_3=S_3,
	\]
	and
	\begin{equation}\label{eq:example-M}
		\M=\add\{A,B,C,D\}.
	\end{equation}
	As representations of $Q$ these are
	\[
	A=(\K\to0\to0),\quad
	B=(\K\xrightarrow{1}\K\to0),\quad
	C=(0\to\K\xrightarrow{1}\K),\quad
	D=(0\to0\to\K).
	\]
	Put $E=S_2=(0\to\K\to0)$. To determine all indecomposable
	modules, take a representation
	$V=(V_1\xrightarrow a V_2\xrightarrow b V_3)$ with $ba=0$.
	The inclusion $\operatorname{im}a\subseteq\ker b$ allows us to
	choose vector-space decompositions
	\[
	V_1=\ker a\oplus U_1,
	\qquad
	V_2=\operatorname{im}a\oplus U_2\oplus W_2,
	\qquad
	V_3=\operatorname{im}b\oplus W_3,
	\]
	where $\ker b=\operatorname{im}a\oplus U_2$.
	The restrictions $a|_{U_1}:U_1\to\operatorname{im}a$ and
	$b|_{W_2}:W_2\to\operatorname{im}b$ are isomorphisms.
	Consequently, $V$ is the direct sum of the five representations
	\[
	(U_1\xrightarrow{\sim}\operatorname{im}a\to0),~
	(0\to W_2\xrightarrow{\sim}\operatorname{im}b),~
	(\ker a\to0\to0),~
	(0\to U_2\to0),~
	(0\to0\to W_3).
	\]
	Writing $r=\operatorname{rank}a$ and $s=\operatorname{rank}b$
	and choosing bases in these summands gives
	\begin{equation}\label{eq:representation-decomposition}
		V\cong B^{\oplus r}\oplus C^{\oplus s}\oplus A^{\oplus(\dim V_1-r)}
		\oplus E^{\oplus(\dim V_2-r-s)}\oplus D^{\oplus(\dim V_3-s)}.
	\end{equation}
	Each of $A,B,C,D,E$ has endomorphism ring $\K$, so these
	five modules are indecomposable. Their dimension vectors are
	distinct, and \eqref{eq:representation-decomposition} proves
	that there are no other indecomposable modules.
	The projective indecomposables are $B,C,D$, and the injective
	indecomposables are $A,B,C$. This cluster-tilting example also
	appears in \cite[Example~2.4]{AHJKPT25}; we give the calculations
	needed here.
	
	Computing maps vertex by vertex gives the following Hom-spaces,
	with sources indexing the rows and targets indexing the columns:
	\begin{equation}\label{eq:hom-table}
		\begin{array}{c|cccc}
			\Hom_\Lambda(-,-)&A&B&C&D\\\hline
			A&\K&0&0&0\\
			B&\K&\K&0&0\\
			C&0&\K&\K&0\\
			D&0&0&\K&\K
		\end{array}
	\end{equation}
	Let $u:D\to C$, $v:C\to B$, and $w:B\to A$ be the maps that
	are identities at vertices $3$, $2$, and $1$, respectively, and
	zero at the other vertices. They give the projective resolution
	\begin{equation}\label{eq:main-2exact}
		0\longrightarrow D\xrightarrow{u}C\xrightarrow{v}B
		\xrightarrow{w}A\longrightarrow0.
	\end{equation}
	The three nonzero maps between distinct indecomposables of $\M$
	are, up to scalar,
	$$D\xrightarrow{u}C\xrightarrow{v}B\xrightarrow{w}A, \qquad vu=wv=0.$$

	They are irreducible in $\M$. The map $v$ factors through $E$
	in $\modcat\Lambda$, so the ambient category matters in this assertion.
	
	We next verify that $\M$ has the required higher-homological structure.
	Recall that a functorially finite subcategory $\mathcal T$ of
	$\modcat\Lambda$ is $2$-cluster tilting if
	\[
	\mathcal T
	=\{Z\mid\Ext^1_\Lambda(Z,T)=0\text{ for all }T\in\mathcal T\}
	=\{Z\mid\Ext^1_\Lambda(T,Z)=0\text{ for all }T\in\mathcal T\}.
	\]
	Functorial finiteness means that every module admits a right and a left
	\(\mathcal T\)-approximation. Every morphism from or to an object of
	\(\mathcal T\) factors through the corresponding approximation.
	
	\begin{proposition}\label{prop:example-cluster}
		The subcategory $\M$ is $2$-cluster tilting in $\modcat\Lambda$.
		Consequently, it is $2$-abelian.
	\end{proposition}
	
	\begin{proof}
		The category $\modcat\Lambda$ is Hom-finite and $\M$ is the additive closure of
		finitely many objects. Taking evaluation maps associated with
		bases of their Hom-spaces gives right and left $\M$-approximations.
		Thus $\M$ is functorially finite. It contains all projectives
		and all injectives, so it is also generating and cogenerating.
		
		Applying $\Hom_\Lambda(-,Y)$ to the projective resolution
		\eqref{eq:main-2exact} gives the cochain complex
		\[
		\Hom_\Lambda(B,Y)\xrightarrow{v^*}\Hom_\Lambda(C,Y)
		\xrightarrow{u^*}\Hom_\Lambda(D,Y)
		\]
		in degrees $0,1,2$. By \eqref{eq:hom-table}, these complexes
		are as follows:
		\begin{equation}\label{eq:ext-complexes}
			\begin{array}{c|c}
				Y&\Hom_\Lambda(B,Y)\longrightarrow\Hom_\Lambda(C,Y)
				\longrightarrow\Hom_\Lambda(D,Y)\\\hline
				A&\K\longrightarrow0\longrightarrow0\\
				B&\K\xrightarrow{\sim}\K\longrightarrow0\\
				C&0\longrightarrow\K\xrightarrow{\sim}\K\\
				D&0\longrightarrow0\longrightarrow\K
			\end{array}
		\end{equation}
		The nonzero differentials are isomorphisms: precomposition sends
		$1_B$ to $v$ in the row $Y=B$, and $1_C$ to $u$ in the row
		$Y=C$. Thus
		\[
		\Ext^1_\Lambda(A,Y)=\ker u^*/\operatorname{im}v^*=0
		\qquad(Y=A,B,C,D).
		\]
		For $Y=A,D$, the degree-one term is zero; for $Y=B$, the
		degree-zero differential is surjective; for $Y=C$, the
		degree-one differential is injective. The other twelve groups
		$\Ext^1_\Lambda(X,Y)$, with $X\in\{B,C,D\}$ and
		$Y\in\{A,B,C,D\}$, vanish because $B,C,D$ are projective.
		Additivity in both variables therefore gives $\Ext^1_\Lambda(\M,\M)=0$.

		It remains to exclude $E$ from both Ext-orthogonal subcategories.
		Maps of representations give
		\[
		\Hom_\Lambda(B,E)=0,
		\qquad \Hom_\Lambda(C,E)=\K,
		\qquad \Hom_\Lambda(D,E)=0.
		\]
		Applying $\Hom_\Lambda(-,E)$ to \eqref{eq:main-2exact}
		therefore gives
		$\Ext^1_\Lambda(A,E)\cong\K$.
		The projective resolution
		\[
		0\longrightarrow D\xrightarrow{u}C\longrightarrow E
		\longrightarrow0
		\]
		and $\Hom_\Lambda(C,D)=0$ give
		\[
		\Ext^1_\Lambda(E,D)
		\cong\operatorname{coker}\bigl(
		\Hom_\Lambda(C,D)\longrightarrow\Hom_\Lambda(D,D)\bigr)
		\cong\K.
		\]
		By \eqref{eq:representation-decomposition}, every module
		$Z\notin\M$ has $E$ as a direct summand. Additivity of
		$\Ext^1$ then gives both $\Ext^1_\Lambda(A,Z)\neq0$ and
		$\Ext^1_\Lambda(Z,D)\neq0$. Since $A,D\in\M$, neither
		Ext-orthogonal subcategory contains such a $Z$. Together with
		$\Ext^1_\Lambda(\M,\M)=0$, this proves the two required
		equalities.
		Since $\M$ is a $2$-cluster tilting subcategory of the abelian
		category $\modcat\Lambda$, it is $2$-abelian by
		\cite[Theorem~3.16]{J16}.
	\end{proof}
	
	For a morphism $f:X\to Y$ in $\M$, a complex
	$K_2\to K_1\to X$ in $\M$ is a $2$-kernel of $f$ if and only if
	\[
	0\longrightarrow K_2\longrightarrow K_1\longrightarrow X
	\xrightarrow{f}Y
	\]
	is exact at $K_2,K_1,X$ in $\modcat\Lambda$. Dually, a complex
	$Y\to C_1\to C_2$ in $\M$ is a $2$-cokernel of $f$ if and only if
	\[
	X\xrightarrow{f}Y\longrightarrow C_1\longrightarrow C_2
	\longrightarrow0
	\]
	is exact at $Y,C_1,C_2$. A four-term complex in $\M$ is $2$-exact
	if and only if adjoining a zero at each end gives an exact sequence
	in $\modcat\Lambda$; see \cite[Proposition~2.9(1)--(3)]{HJ21}.
	Thus \eqref{eq:main-2exact}, with its external zeros omitted, is
	$2$-exact. The minimal $2$-kernels and $2$-cokernels of $u,v,w$ are
	\begin{equation}\label{eq:minimal-table}
		\begin{array}{c|c|c}
			f&\text{minimal $2$-kernel}&\text{minimal $2$-cokernel}\\\hline
			u:D\to C&0\to0\to D&C\xrightarrow{v}B\xrightarrow{w}A\\
			v:C\to B&0\to D\xrightarrow{u}C&B\xrightarrow{w}A\to0\\
			w:B\to A&D\xrightarrow{u}C\xrightarrow{v}B&A\to0\to0
		\end{array}
	\end{equation}
	The maps between the two kernel terms or the two cokernel terms
	are radical, as required by Definition~\ref{def:minimal}.
	
	The degree-two cohomology in \eqref{eq:ext-complexes} also gives
	\begin{equation}\label{eq:ext2-example}
		\Ext^2_\Lambda(X,Y)\cong
		\begin{cases}
			\K,&(X,Y)=(A,D),\\
			0,&\text{otherwise},
		\end{cases}
		\qquad X,Y\in\{A,B,C,D\}.
	\end{equation}
	The class of \eqref{eq:main-2exact} corresponds to the image of
	$1_D$ in
	\[
	\Hom_\Lambda(D,D)/
	\operatorname{im}\bigl(\Hom_\Lambda(C,D)\to\Hom_\Lambda(D,D)\bigr)
	=\K.
	\]
	It is therefore nonzero and generates $\Ext^2_\Lambda(A,D)$.
	Its middle map $v$ is radical, so the sequence is a minimal
	$2$-extension.
	
	The minimal complexes and extension groups determine the
	wide-subcategory lattice.
	
	\begin{proposition}\label{prop:example-wide-classification}
		The wide subcategories of $\M$ are exactly
		\begin{equation}\label{eq:8wide}
			\mathbf0,
			\quad
			\add A,\ \add B,\ \add C,\ \add D,
			\quad
			\add(A\oplus C),\ \add(B\oplus D),
			\quad
			\M.
		\end{equation}
	\end{proposition}
	
	\begin{proof}
		Let $\W$ be wide. By Lemma~\ref{lem:minimal-characterization},
		the minimal $2$-kernel of $w:B\to A$ in
		\eqref{eq:minimal-table} gives
		\[
		A,B\in\W\quad\Longrightarrow\quad C,D\in\W.
		\]
		The minimal $2$-kernel and $2$-cokernel of $v:C\to B$ give
		\[
		B,C\in\W\quad\Longrightarrow\quad A,D\in\W,
		\]
		and the minimal $2$-cokernel of $u:D\to C$ gives
		\[
		C,D\in\W\quad\Longrightarrow\quad A,B\in\W.
		\]
		Finally, if $A,D\in\W$, the minimal $2$-extension
		\eqref{eq:main-2exact} gives $B,C\in\W$ by
		Lemma~\ref{lem:minimal-characterization}. Thus any wide
		subcategory containing one of the pairs
		\begin{equation}\label{eq:example-forcing-pairs}
			\{A,B\},\quad \{B,C\},\quad \{C,D\},\quad \{A,D\}
		\end{equation}
		is all of $\M$.
		
		To verify the remaining candidates, first consider either
		$\W=\add(A\oplus C)$ or $\W=\add(B\oplus D)$.
		Write $(S,T)=(A,C)$ in the first case and $(S,T)=(B,D)$
		in the second. In both cases, \eqref{eq:hom-table} gives
		\[
		\Hom_\Lambda(S,T)=\Hom_\Lambda(T,S)=0,
		\qquad
		\End_\Lambda(S)=\End_\Lambda(T)=\K.
		\]
		Every morphism
		\[
		f:S^{\oplus a}\oplus T^{\oplus b}
		\longrightarrow S^{\oplus c}\oplus T^{\oplus d}
		\]
		is therefore a block diagonal matrix with one $c\times a$
		matrix and one $d\times b$ matrix over $\K$. Independent
		changes of basis in these two blocks put each matrix into its
		rank normal form. Consequently, there are decompositions of
		its source and target as $I\oplus X_0$ and $I\oplus Y_0$,
		with $I,X_0,Y_0\in\W$, for which
		\[
		f=\begin{pmatrix}1_I&0\\0&0\end{pmatrix}:
		I\oplus X_0\longrightarrow I\oplus Y_0.
		\]
		For this form of $f$, a $2$-kernel and a $2$-cokernel in $\M$
		are, respectively,
		\[
		0\longrightarrow X_0\xrightarrow{\binom{0}{1}}
		I\oplus X_0,
		\qquad
		I\oplus Y_0\xrightarrow{(0\;1)}Y_0\longrightarrow0.
		\]
		Indeed, applying $\Hom_\Lambda(Z,-)$ or
		$\Hom_\Lambda(-,Z)$ for any $Z\in\M$ gives the required
		exact sequences, since the displayed inclusion and projection
		split and $f$ is the identity on $I$. Transporting these
		complexes along the chosen source and target isomorphisms
		proves condition~\ref{W1} for every morphism in each of the
		two subcategories.
		
		For $X,Y\in\add(A\oplus C)$, every summand of
		$\Ext^2_\Lambda(Y,X)$ obtained by additivity in both variables
		is one of
		\[
		\Ext^2_\Lambda(A,A),\quad \Ext^2_\Lambda(A,C),\quad
		\Ext^2_\Lambda(C,A),\quad \Ext^2_\Lambda(C,C),
		\]
		and all four vanish by \eqref{eq:ext2-example}.
		For $X,Y\in\add(B\oplus D)$, the module $Y$ is projective,
		so $\Ext^2_\Lambda(Y,X)=0$ as well.
		By \cite[Proposition~2.9(4)]{HJ21}, equivalence classes of
		$2$-extensions of $Y$ by $X$ in $\M$ are identified with
		$\Ext^2_\Lambda(Y,X)$. Thus every such extension with endpoints
		in either of the two subcategories is equivalent to
		\[
		0\longrightarrow X\xrightarrow{1_X}X\xrightarrow{0}Y
		\xrightarrow{1_Y}Y\longrightarrow0.
		\]
		The middle terms lie in the same subcategory as the endpoints,
		which proves condition~\ref{W2}. Hence both
		$\add(A\oplus C)$ and $\add(B\oplus D)$ are wide.
		
		For $\add Z$ with $Z\in\{A,B,C,D\}$, the same verification
		uses a single matrix block and the equality
		$\Ext^2_\Lambda(Z,Z)=0$. The zero subcategory satisfies the
		two conditions using zero complexes, and $\M$ itself is wide.
		Finally, closure under direct summands and
		\eqref{eq:representation-decomposition} show that every wide
		subcategory of $\M$ is determined by the subset of
		$\{A,B,C,D\}$ it contains. Every proper subset with at least
		three elements contains one of the pairs in
		\eqref{eq:example-forcing-pairs}. Among the two-element subsets,
		only $\{A,C\}$ and $\{B,D\}$ avoid these pairs. This exhausts
		all sixteen subsets and proves \eqref{eq:8wide}.
	\end{proof}
	
	The resulting lattice is shown in Figure~\ref{fig:wide-lattice-example}.
	\begin{figure}[htbp]
		\centering
		\begin{tikzpicture}[>=latex, every node/.style={font=\small}]
			\node (M) at (0,3.6) {$\M$};
			\node (AC) at (-2.1,2.4) {$\add(A\oplus C)$};
			\node (BD) at (2.1,2.4) {$\add(B\oplus D)$};
			\node (A) at (-3.0,1.2) {$\add A$};
			\node (C) at (-1.2,1.2) {$\add C$};
			\node (B) at (1.2,1.2) {$\add B$};
			\node (D) at (3.0,1.2) {$\add D$};
			\node (Z) at (0,0) {$\mathbf0$};
			\draw[->] (A)--(AC); \draw[->] (C)--(AC);
			\draw[->] (B)--(BD); \draw[->] (D)--(BD);
			\draw[->] (AC)--(M); \draw[->] (BD)--(M);
			\draw[->] (Z)--(A); \draw[->] (Z)--(B); \draw[->] (Z)--(C); \draw[->] (Z)--(D);
		\end{tikzpicture}
		\caption{The lattice $\Wide_2(\M)$.}
		\label{fig:wide-lattice-example}
	\end{figure}
	
	We next compute the prime spectrum, supports, and radicals. Since the lattice is finite, a proper wide subcategory is a prime wide subcategory exactly when it has a unique upper cover.  Therefore
	\begin{equation}\label{eq:example-primes}
		\Prime_2(\M)
		=\{\add A,\add B,\add C,\add D,
		\add(A\oplus C),\add(B\oplus D)\}.
	\end{equation}
	Write
	\[
	p_A=\add A,\quad p_B=\add B,\quad p_C=\add C,\quad p_D=\add D,
	\]
	\[
	q_{AC}=\add(A\oplus C),
	\qquad
	q_{BD}=\add(B\oplus D).
	\]
	Then
	\[
	\overline{\{q_{AC}\}}=\{p_A,p_C,q_{AC}\},
	\qquad
	\overline{\{q_{BD}\}}=\{p_B,p_D,q_{BD}\},
	\]
	while each singleton $\{p_\bullet\}$ is closed. Thus the prime spectrum is the disjoint union of two three-point $V$-shaped finite $T_0$ spaces.
	
	The supports of the four indecomposables are
	\[
	\begin{gathered}
		\Supp(A)=\{p_B,p_C,p_D,q_{BD}\},\qquad
		\Supp(B)=\{p_A,p_C,p_D,q_{AC}\},\\
		\Supp(C)=\{p_A,p_B,p_D,q_{BD}\},\qquad
		\Supp(D)=\{p_A,p_B,p_C,q_{AC}\}.
	\end{gathered}
	\]
	Furthermore,
	\[
	\Supp(\add(A\oplus C))
	=\Prime_2(\M)\setminus\{q_{AC}\},
	\]
	\[
	\Supp(\add(B\oplus D))
	=\Prime_2(\M)\setminus\{q_{BD}\},
	\]
	and
	\[
	\Supp(\mathbf0)=\varnothing,
	\qquad
	\Supp(\M)=\Prime_2(\M).
	\]
	These eight supports are pairwise distinct.

	\begin{proposition}\label{prop:example-radical}
		Every wide subcategory of $\M$ is $\mathfrak P$-radical, where $\mathfrak P=\Prime_2(\M)$.
	\end{proposition}
	
	\begin{proof}
		This follows from Proposition~\ref{prop:prime-separation}.
		Explicitly, each of the six nonzero proper wide subcategories is
		itself a prime wide subcategory, the intersection of all six
		prime wide subcategories is $\mathbf0$,
		and the radical of $\M$ is the empty intersection $\M$.
	\end{proof}
	
	The prime spectrum is a finite $T_0$-space, hence noetherian and sober.
	Nevertheless, its canonical support is not classifying: the support
	assignments fail to realize all specialization-closed subsets as supports
	of wide subcategories. Indeed, a component with closed points $p,p'$
	and generic point $q$ has precisely the specialization-closed subsets
	\[
	\varnothing,\quad \{p\},\quad \{p'\},\quad \{p,p'\},
	\quad \{p,p',q\}.
	\]
	The choices on the two components are independent, so
	\[
	\left|\Spcl\!\left(\Spec_{\Prime_2(\M)}(\M)\right)\right|=25,
	\]
	whereas Proposition~\ref{prop:example-wide-classification} gives
	\[
	|\Wide_2(\M)|=8.
	\]
	For a concrete unrealized closed subset, take $\{p_A\}$.
	Since $p_B$ and $p_C$ occur among the other prime wide subcategories
	and $p_B\cap p_C=\mathbf0$, \eqref{eq:inverse-intersection}
	gives
	\[
	\Supp^{-1}(\{p_A\})=\mathbf0,
	\qquad \Supp(\Supp^{-1}(\{p_A\}))=\varnothing\neq\{p_A\}.
	\]
	Thus $\{p_A\}$ is not the support of any wide subcategory.
	
	\begin{corollary}\label{cor:no-classifying-example}
		The $2$-abelian category $\M$ in \eqref{eq:example-M} admits no classifying $2$-support datum in the sense of Definition~\ref{def:classifying}.
	\end{corollary}
	
	\begin{proof}
		If $(Y,\sigma)$ were classifying, Theorem~\ref{thm:reconstruction} would give a homeomorphism
		\[
		Y\cong\Spec_{\Prime_2(\M)}(\M).
		\]
		Hence $|\Spcl(Y)|=25$.  But the classifying property would also give a lattice isomorphism
		\[
		\Wide_2(\M)\cong\Spcl(Y),
		\]
		contradicting $|\Wide_2(\M)|=8$.
	\end{proof}
	
	\begin{remark}
		Proposition~\ref{prop:prime-separation} shows that prime-radicality
		holds for every wide subcategory under our general hypotheses.
		The obstruction in this example is the failure to realize all
		specialization-closed subsets as supports of wide subcategories. It is also visible directly
		in the lattice: if $\mathcal A=\add A$, $\mathcal B=\add B$,
		and $\mathcal C=\add C$, then
		\[
		\mathcal A\wedge(\mathcal B\vee\mathcal C)=\mathcal A,
		\qquad
		(\mathcal A\wedge\mathcal B)\vee(\mathcal A\wedge\mathcal C)=\mathbf0.
		\]
		Thus $\Wide_2(\M)$ is not distributive, whereas the lattice of
		specialization-closed subsets of any space is distributive.
	\end{remark}

\end{document}